\documentclass[10pt]{amsart}
\usepackage{float}
\usepackage{tikz-cd}
\usetikzlibrary{cd}

\usepackage{amsmath, amssymb, amsthm, wasysym, xcolor, comment}
\usepackage[margin=2cm]{geometry}
\usepackage{graphicx}
\usepackage{manyfoot}
\usepackage{booktabs}
\usepackage{parskip}
\usepackage{setspace}
\usepackage{enumitem}
\usepackage{caption}
\usepackage{subcaption}

\usepackage[colorlinks=true, allcolors=blue]{hyperref}

\makeatletter
\renewcommand\subsubsection{\@secnumfont}{\bfseries}%
\renewcommand\subsubsection{\@startsection{subsubsection}{3}
  \z@{.5\linespacing\@plus.7\linespacing}{-.5em}%
  {\normalfont\bfseries}}
  \makeatother

\theoremstyle{definition}
\theoremstyle{definition}
\newtheorem{theorem}{Theorem}[section]

\newtheorem{proposition}[theorem]{Proposition}
\newtheorem{lemma}[theorem]{Lemma}

\newtheorem{corollary}[theorem]{Corollary}

\newcommand{\Z}{\mathbb{Z}}

\DeclareMathOperator{\lcm}{\lcm}

\title{The Non-Orientable Four-Ball Genus of a New Infinite Family of Torus Knots}
\author{Shreya Sinha}
\date{July 2025}

\begin{document}

\begin{abstract}
    We extend previous work by using a combination of band surgeries and known bounds to compute $\gamma_4(T_{4n, (2n\pm1)^2 + 4n-2}) = 2n-1$ for all $n \geq 1$. We further generalize this result by showing that $\gamma_4(T_{4n + 2k, n(4n + 2k) - 1}) = \gamma_4(T_{4n + 2k, (n+2)(4n + 2k) - 1}) = 2n-1 + k$ for all $n \geq 1$ and $k \geq 0$. All knots in this family, with the exception of $T_{4, 3}$, are counterexamples to Batson's conjecture.
\end{abstract}

\maketitle

\section{Introduction}
\subsection*{Preliminaries}
Given a knot $K$, one measure of its complexity is the constraints it places on the complexity of surfaces that bound it. In particular, if we restrict to smooth, non-orientable surfaces that are properly embedded in $B^4$, then the minimal genus among such surfaces provides a particularly informative measure. This invariant is known as the \textit{smooth non-orientable four-genus} of $K$, introduced by Murakami and Yasuhara in 2000 \cite{murakami2000FourgenusAF}, and denoted as $$\gamma_4(K) = \min \left\{\dim(H_1(F; \Z / 2\Z)) \mid F \text{ non-orientable surface}, \space \partial(F) = K, \space F \hookrightarrow B^4 \right\}$$ In the case that $K$ is slice, we define $\gamma_4(K ) = 0$. 

The non-orientable four-genus has been computed for several families of knots. The values for knots with 8 or 9 crossings were determined by Jabuka and Kelly \cite{Jabuka_2018}, those with 10 crossings by Ghanbarian \cite{Ghanbarian_2022}, and most 11 crossing non-alternating knots by Fairchild \cite{Fairchild_2024}. For many double twist knots, the non-orientable four-genus has been computed by Hoste, Shanahan, and Van Cott \cite{hoste2023nonorientable4genusdoubletwist}. A particularly active area of study involves torus knots. Allen computed $\gamma_4(T_{p, q})$ for all $p = 2, 3$ \cite{allen2020nonorientablesurfacesboundedknots}; Binns, Kang, Simone, and Truöl computed most cases for $p = 4$ \cite{binns2024nonorientablefourballgenustorus}, and Fairchild, Garcia, Murphy, and Percle have computed many cases for $p = 5, 6$ \cite{fairchild2024nonorientable}.

A particularly useful upper bound for $\gamma_4(T_{p, q})$ is given by the number of \textit{pinch moves} — band moves between adjacent strands relative to a standard torus knot diagram — required to convert $T_{p, q}$ into the unknot. This number is denoted by $\vartheta(T_{p, q})$. In \cite{batson2012nonorientablefourballgenusarbitrarily}, Batson conjectured that $\gamma_4(T_{p, q}) = \vartheta(T_{p, q})$ and proved this equality holds for the infinite family of torus knots of the form $T_{2k, 2k-1}$ with $k \geq 1$, establishing in particular that $\gamma_4(T_{2k, 2k-1}) = \vartheta(T_{2k, 2k-1}) = k - 1$. However, Lobb gave a counterexample to Batson's conjecture in \cite{lobb2019counterexamplebatsonsconjecture}, showing that for $T_{4, 9}$, we have $1 = \gamma_4(T_{4, 9}) = \vartheta(T_{4, 9}) -1$ by employing a non-pinch band move that sends $T_{4, 9}$ to the Stevedore knot $6_1$.

In unpublished work, Tairi \cite{tairi2020} provided another counterexample with the torus knot $T_{4, 11}$, proving that $1 = \gamma_4(T_{4, 11}) = \vartheta(T_{4, 11}) - 1$ via a non-pinch band move that also reduces $T_{4, 11}$ to $6_1$. Tairi further computed that $\gamma_4 (T_{4m + 4, 12mb + 6b + 5 - 2m}) = b$ for all $b \geq 2$ and $m \geq 0$, with $\vartheta(T_{4m + 4, 12mb + 6b + 5 - 2m}) = b+1$. Longo \cite{longo2020infinitefamilycounterexamplesbatsons} generalized Tairi and Lobb's examples by showing $ \gamma_4(T_{4n, (2n\pm 1)^2 })\leq 2n-1$ while $\vartheta(T_{4n, (2n\pm 1)^2 }) = 2n$ for all $n \geq 2$. His examples were primarily inspired by the knot $T_{8, 25}$, for which he found three band moves that turn $T_{8, 25}$ into the slice knot $10_3$. Binns, Kang, Simone, and Truöl \cite{binns2024nonorientablefourballgenustorus} extended Longo's work by showing that $ \gamma_4(T_{4n, (2n\pm 1)^2 })= 2n-1$ when $n \leq 2$ even. They also showed that for even $n \geq 2$ and $k \geq 0$, that $\gamma_4(T_{4n + 2k, (4n+2k)(n\pm 1) +1}) = 2n+k-1 = \vartheta(T_{4n + 2k, (4n+2k)(n\pm 1) +1}) - 1$. 

\subsection*{Results}
In this paper, we extend previous work by  using a combination of band surgeries and known bounds to prove the following theorem. 

\begin{theorem}
    The torus knots $T_{4n, (2n\pm 1)^2 + 4n-2}$ for $n \geq 2$ are counterexamples to Batson's conjecture. In particular, 
    \begin{enumerate}
        \item $\gamma_4(T_{4n, (2n\pm 1)^2 + 4n-2}) = 2n-1$ 
        \item $\vartheta(T_{4n, (2n\pm 1)^2 + 4n-2}) = 2n$. 
    \end{enumerate}

    This also holds for the knot $T_{4, 11}$, which lies in the family $T_{4n, (2n + 1)^2 + 4n-2}$ when $n = 1$. 
\end{theorem}

We further generalize this result to the following: 

\begin{corollary}
    $\gamma_4(T_{4n + 2k, n(4n + 2k) - 1}) = \gamma_4(T_{4n + 2k, (n+2)(4n + 2k) - 1}) = 2n-1+k$ for all $n \geq 1$ and $k \geq 0$. 
\end{corollary}

Consequently, all knots in this family are counterexamples to Batson's conjecture.  

\subsection*{Acknowledgments}

The author would like to express her sincere gratitude to Peter Ozsváth for his invaluable guidance and generosity with his time throughout this research. The author would additionally like to thank Andrew Lobb, Ollie Thakar, and Fraser Binns for insightful discussions that contributed significantly to the development of this work. Finally, the author thanks Cornelia Van Cott and Stanislav Jabuka for valuable comments and correspondences. 

\section{Tools}

In this section, we lay out the tools necessary for our calculations. We first have the following lemma from Jabuka and Van Cott \cite{Jabuka_2021}
\begin{lemma} \label{lem_pinch_move}
    Let $p$ and $q$ be relatively prime positive integers. After applying a pinch move to $T_{p,q}$, the resulting torus knot (up to orientation) is $T_{|p - 2 t|, |q - 2 h|}$, where $t$ and $h$ are the integers uniquely determined by the requirements 
    \begin{equation*}
        t \equiv -q^{-1} \pmod{p} \quad  \text{ where }  t \in  \{0, \dots , p - 1\}
    \end{equation*}
    \begin{equation*}
        h \equiv p^{-1} \pmod{q} \quad \text{ and } h \in \{0, \dots , q - 1\}
    \end{equation*}
\end{lemma}

Using Lemma \ref{lem_pinch_move}, Binns, Kang, Simone and Truöl in \cite{binns2024nonorientablefourballgenustorus},  calculated the pinch number for the following torus knot families: 

\begin{lemma}(Lemma 2.9 in \cite{binns2024nonorientablefourballgenustorus}) \label{pinch-move-calculation-theorem} Let $p \geq 2, k \geq 1$. Performing a non-oriented band move on the torus knot $T_{p, kp \pm 1}$ yields the torus knot $T_{p-2, k(p-2) \pm 1}$. Consequently,
\begin{equation}
    \vartheta(T_{p, kp \pm 1}) = \begin{cases}
        \frac{p-1}{2} & \text{if } p \text{ is odd} \\
        \frac{p}{2} & \text{if } p \text{ is even and } kp\pm 1 \neq p-1 \\
        \frac{p-2}{2} & \text{if } p \text{ is even and } kp\pm 1 = p-1 \\
    \end{cases}
\end{equation}
\end{lemma}

Using Lemma \ref{pinch-move-calculation-theorem}, tools from involutive knot Floer homology, and previous bounds, they found the following narrow bounds on $\gamma_4$: 
\begin{theorem}[Theorem 1.7 in \cite{binns2024nonorientablefourballgenustorus}] \label{1.7_binns}
    Fix $p > 3$. If $q \equiv p-1, p+1$, or $ 2p-1 \pmod{2p}$, then $\vartheta(T_{p, q} ) - 1 \leq \gamma_4(T_{p, q}) \leq \vartheta(T_{p, q})$. In particular, we have the following: 
    \begin{enumerate}[nosep]
        \item[(i)] Let $p$ be odd. 
            \begin{itemize}
                \item If $q \equiv p - 1 \pmod{2p}$, then $\gamma_4(T_{p, q}) \in \{\frac{p-3}{2}, \frac{p-1}{2}\}$. 
                \item If $q \equiv p + 1$ or $ 2p-1 \pmod{2 p}$, then $\gamma_4(T_{p, q}) = \frac{p-1}{2}$. 
            \end{itemize}
        \item[(ii)] Let $p$ be even. 
            \begin{itemize}
                \item If $q > p$ and $q \equiv p - 1, p+1 $ or $2p-1\pmod{2p}$, then $\gamma_4(T_{p, q}) \in \{\frac{p-2}{2}, \frac{p}{2}\}$
            \end{itemize}
    \end{enumerate}
\end{theorem}

We also have the following proposition from Jabuka and Kelly in \cite{Jabuka_2018}.  
\begin{proposition}(Proposition 2.4 in \cite{Jabuka_2018}) \label{prop_bound} If the knots $K$ and $K'$ are related by a non-oriented band move, then 
\begin{equation}
    \gamma_4(K) \leq \gamma_4(K') + 1. 
\end{equation}
If a knot $K$ is related to a slice knot $K'$ by a non-oriented band move, then $\gamma_4(K) = 1$. 

\end{proposition}

\section{Proof of Main Theorem}

\begin{theorem}\label{thm_main}
    The torus knots $T_{4n, (2n\pm 1)^2 + 4n-2}$ for $n\geq 2$ are counterexamples to Batson's conjecture. In particular, 
    \begin{enumerate}
        \item $\gamma_4(T_{4n, (2n\pm 1)^2 + 4n-2}) = 2n-1$ 
        \item $\vartheta(T_{4n, (2n\pm 1)^2 + 4n-2}) = 2n$. 
    \end{enumerate}

    This also holds for the knot $T_{4, 11}$, which lies in the family $T_{4n, (2n + 1)^2 + 4n-2}$ for $n = 1$. 
\end{theorem}
We prove (2) first, and then (1). 

\begin{proof}[Proof of (2)]
    We first calculate the pinch number of all such knots. The knots are either of the form $T_{4n, (2n + 1)^2 + 4n-2} = T_{4n, 4n^2 + 8n - 1}$ or $T_{4n, (2n - 1)^2 + 4n-2} = T_{4n, 4n^2 - 1}$. For both sets of $T_{p, q}$, we have $q > p$ and $q \equiv -1 \pmod{p}$. By Lemma \ref{pinch-move-calculation-theorem}, $\vartheta(T_{4n, (2n\pm 1)^2 + 4n-2}) = 2n$.
\end{proof}

\begin{proof}[Proof of (1)]
    By Theorem \ref{1.7_binns}, we know $\gamma_4(T_{4n, (2n\pm 1)^2 + 4n-2}) \in \{2n-1, 2n\}$. We will show  $$\gamma_4(T_{4n, (2n\pm 1)^2 + 4n-2}) = 2n-1. $$ 

    \begin{figure}[ht!]
    \centering

    \begin{subfigure}{0.40\textwidth}
        \centering
        \includegraphics[width=\linewidth]{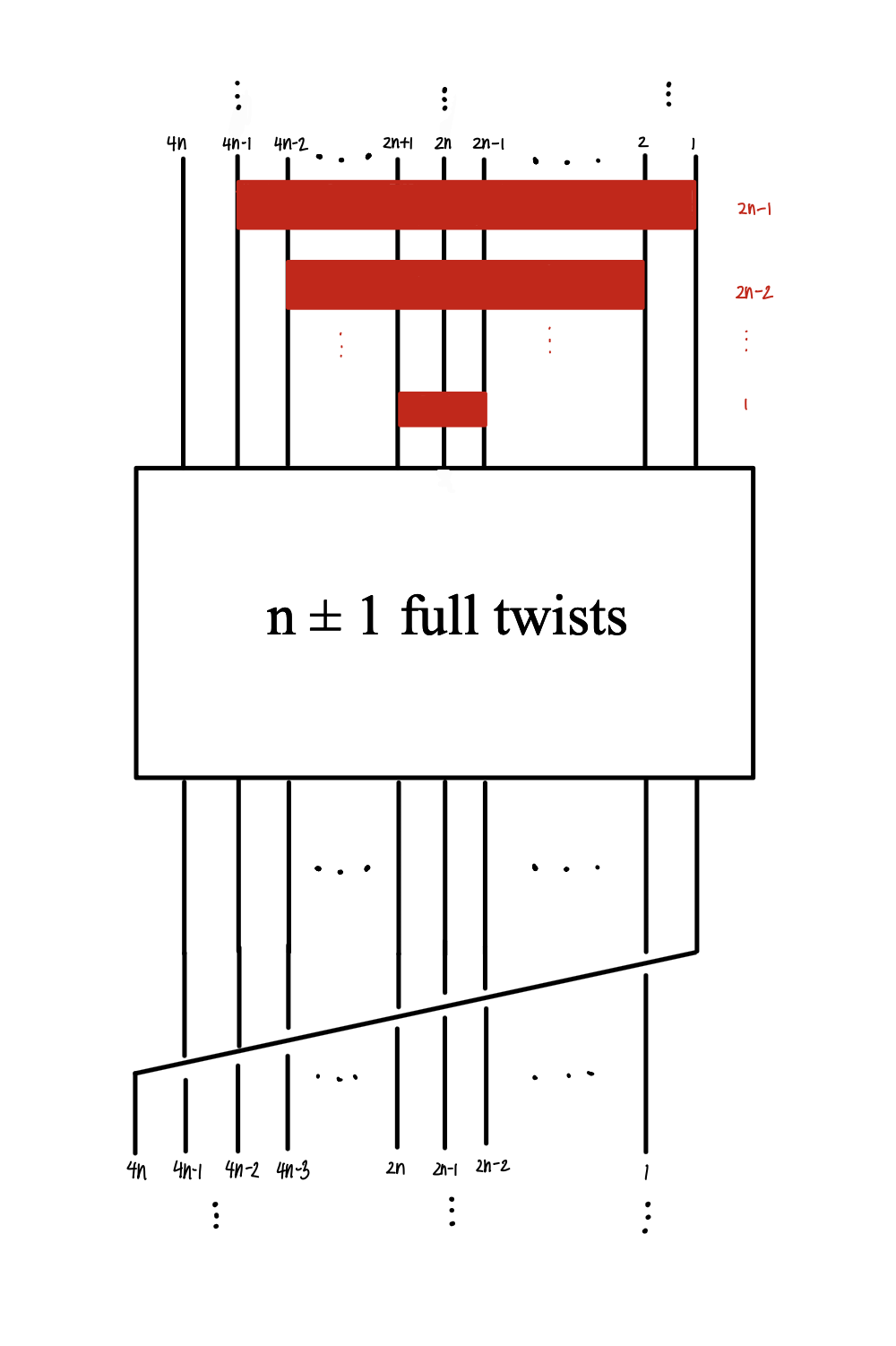}
        \caption{Figure 3 in \cite{longo2020infinitefamilycounterexamplesbatsons}, where $m= n\pm 1$.}
        \label{fig:1}
    \end{subfigure}
    \hspace{0.05\textwidth}
    \begin{subfigure}{0.40\textwidth}
        \centering
        \includegraphics[width=\linewidth]{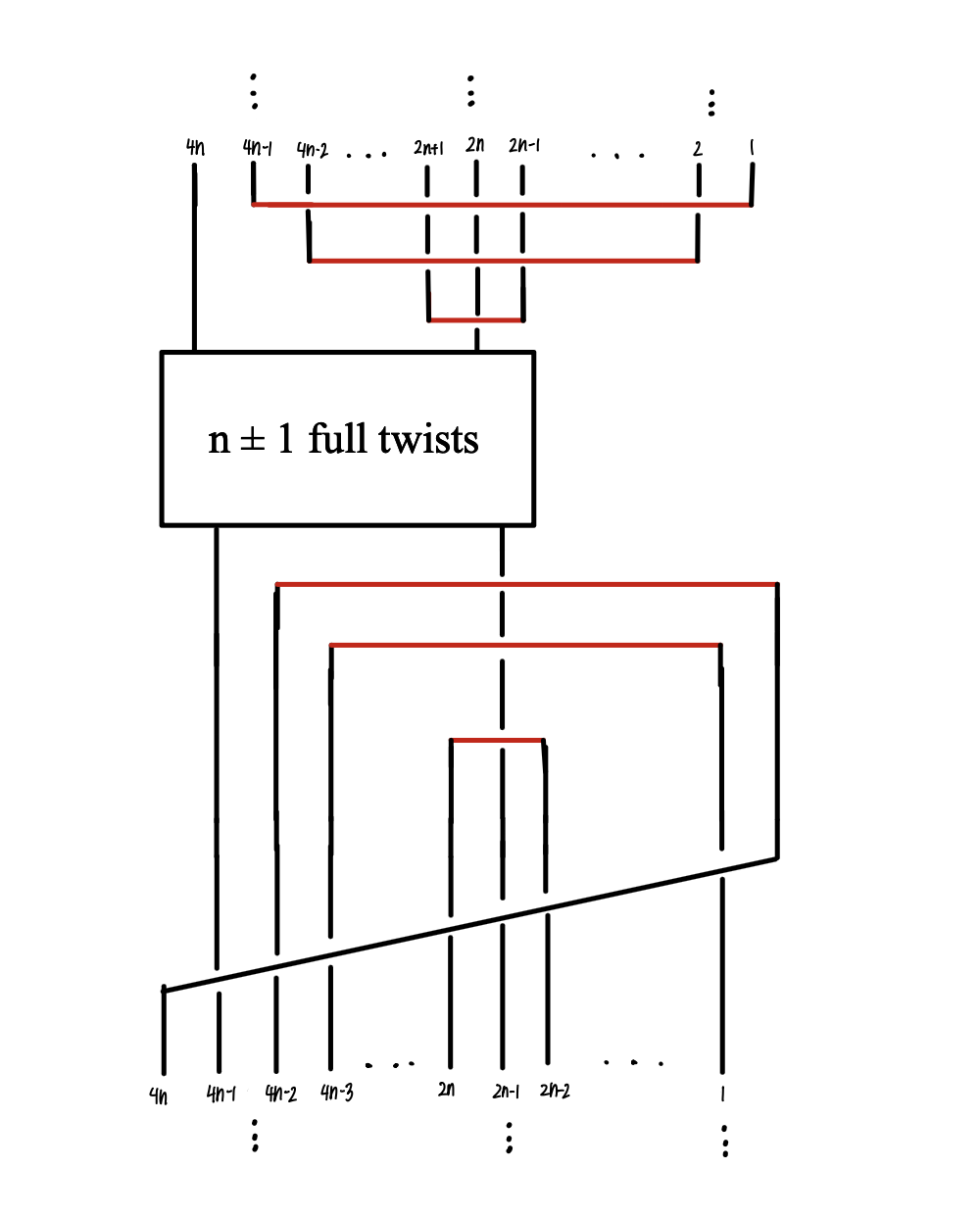}
        \caption{Pulling the joined strands through the $n \pm 1$ full twists. }
        \label{fig:2}
    \end{subfigure}

    \caption{}
    \label{fig_1_and_2}
    \end{figure}    
    
    To do so, we will perform a sequence of unoriented band surgeries on $T_{4n, (2n\pm 1)^2 + 4n-2}$ and demonstrate that the resulting knot is isotopic to the one obtained by Longo's choice of band surgeries on  $T_{4n, (2n\pm 1)^2}$. We use the term \textit{partial twist} to refer to a single strand crossing over all other strands in a torus knot with respect to the standard diagram. For example, each full twist for the knot $T_{p, q}$ consists of $p$ partial twists. Let us now review Longo's choice of bands in \cite{longo2020infinitefamilycounterexamplesbatsons}. The knot $T_{4n, (2n\pm 1)^2}$ has $\frac{(2n\pm 1)^2 - 1 }{4n} = n \pm 1$ full twists with a single additional partial twist.

    For simplicity, we will refer to the position of a strand relative to other strands. For example, in Figure \ref{fig:1}, we label the $4n$ strands in increasing order from right to left. This labeling resets after each partial twist, so that the rightmost strand is always labeled $1$. Performing surgery on each band in Figure \ref{fig:1}, we can isotope the resulting diagram by pulling joined strands through the full twists, yielding Figure \ref{fig:2}. 

    \begin{figure}[ht!]
    \centering

    \begin{subfigure}{0.45\textwidth}
        \centering
        \includegraphics[width=\linewidth]{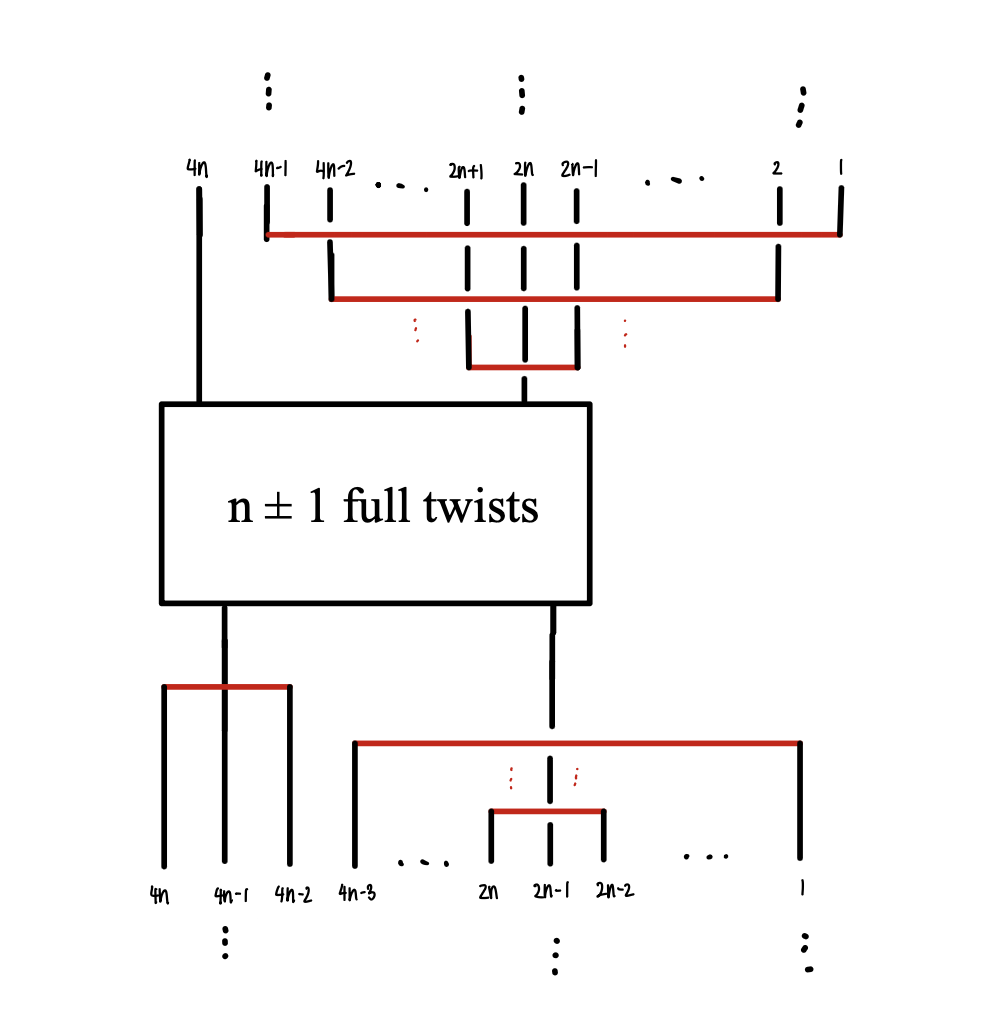}
        \caption{The knot $K_1$, obtained by pulling joined strands through the final partial twist. }
        \label{fig:3}
    \end{subfigure}
    \hfill
    \begin{subfigure}{0.45\textwidth}
        \centering
        \includegraphics[width=\linewidth]{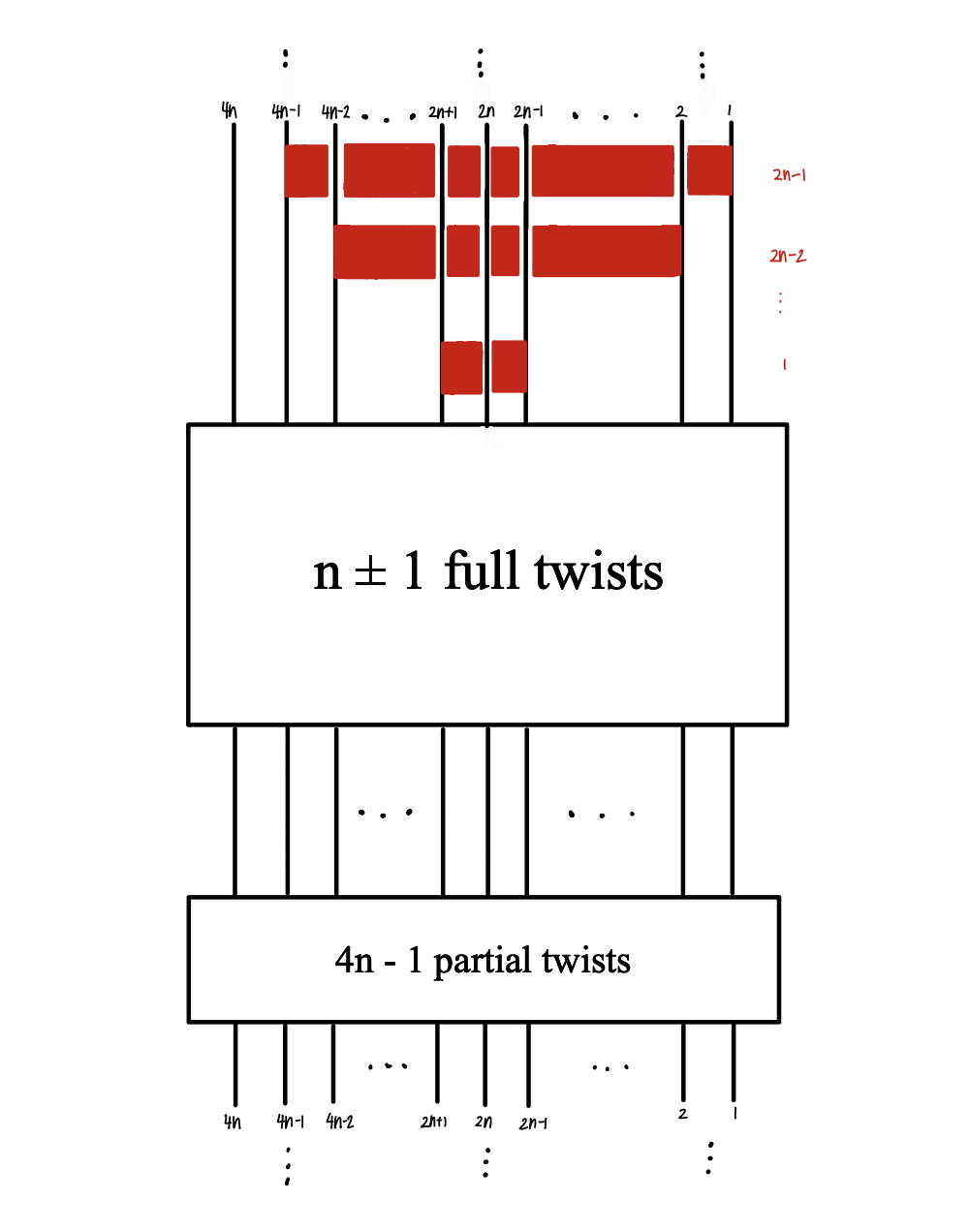}
        \caption{$2n-1 $ bands attached to $T_{4n, (2n\pm 1)^2 + 4n-2}$.}
        \label{fig:4}
    \end{subfigure}

    \caption{}
    \label{fig:test}
    \end{figure}

    Notice that, under this labeling scheme, the $i$th strand is joined to the $4n-i$th strand (the join is highlighted in red in each figure) for all $1 \leq i \leq 2n-1$. Let us refer to such a pair of strands and their join as the $(i, 4n-i)$ \textit{pair}. After every partial twist, the $(i, 4n-i)$ pair becomes the $(i + 1 \pmod {4n}, 4n-i + 1 \pmod {4n})$ pair. The final isotopy comes from pushing the pairs through the last partial twist, resulting in Figure \ref{fig:3}, where the $i$th and $4n-2-i$'th strands are paired. Let this knot be denoted as $K_1$. 

    We now examine the effect of a similar set of band surgeries on $T_{4n, (2n \pm 1)^2 + 4n-2}$. Our goal is to show that the resulting knot is isotopic to $K_1$. Consider the set of bands attached to $\gamma_4(T_{4n, (2n\pm 1)^2 + 4n-2})$ depicted in Figure \ref{fig:4}. Note that, in contrast to Figure \ref{fig:1}, where all the bands are attached over the strands relative to the direction of the twists, the bands in Figure \ref{fig:4} are positioned below the knot relative to the same direction of the twists. After performing the band surgeries, we isotope the diagram in Figure \ref{fig:4} to get Figure \ref{fig:5} by pulling each pair through the $n \pm 1$ partial twists.  

    \begin{figure}[ht!]
    \centering
     \begin{subfigure}{0.45\textwidth}
        \centering
        \includegraphics[width=\linewidth]{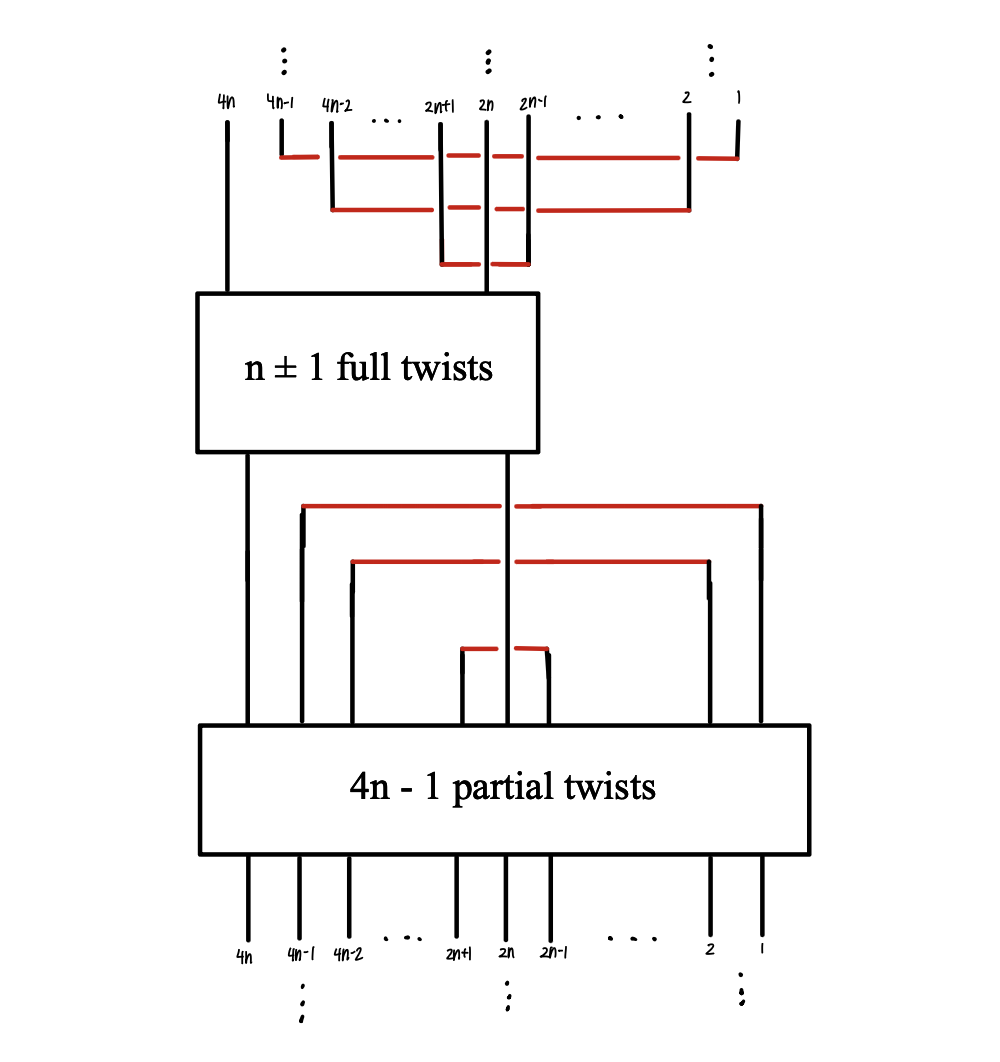}
        \caption{After an isotopy of Figure \ref{fig:4}.} 
        \label{fig:5}
    \end{subfigure}
    \hfill
    \begin{subfigure}{0.45\textwidth}
        \centering
        \includegraphics[width=\linewidth]{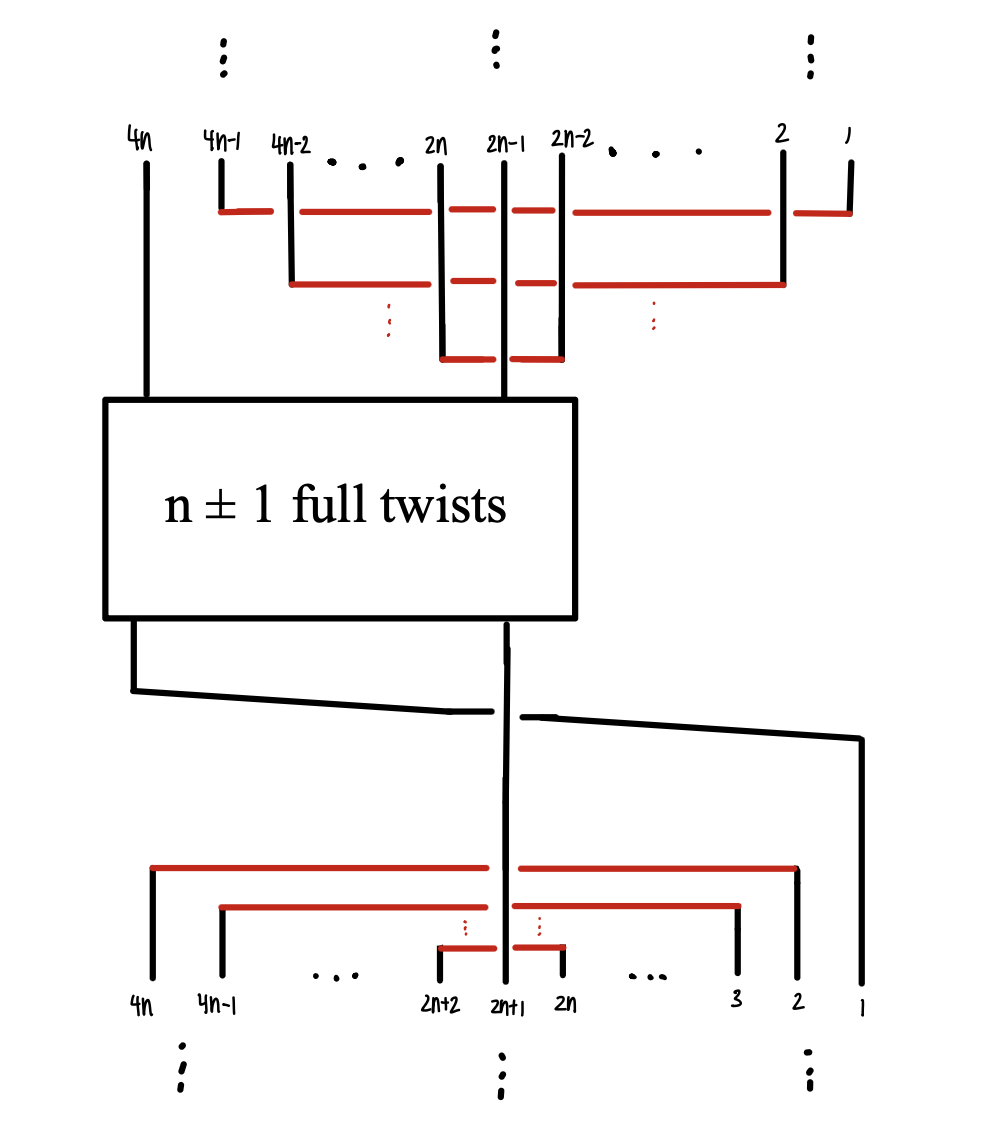}
        \caption{After pulling the pairs through the $4n-1$ partial twists.}
        \label{fig:6}
    \end{subfigure}

    \caption{}
    \label{fig_5_and_6}
    \end{figure}

    Similar to Figure \ref{fig:2} from the first set of torus knots,  the strands $i$ and $4n-i$ are paired, denoted in red. Pushing these pairs through the $4n - 1 $ partial twists transforms the $(i, 4n-i)$ pair to the $(i - (4n-1) \pmod{4n}, 4n-i - (4n -1) \pmod{4n}) = (i + 1 \pmod{4n}, 4n-i +1 \pmod{4n})$ pair. 
    For $2 \leq i \leq 2n$, this operation effectively shifts each band down by one position. 

    \begin{figure}[ht!]
        \centering
        \includegraphics[width=0.5\linewidth]{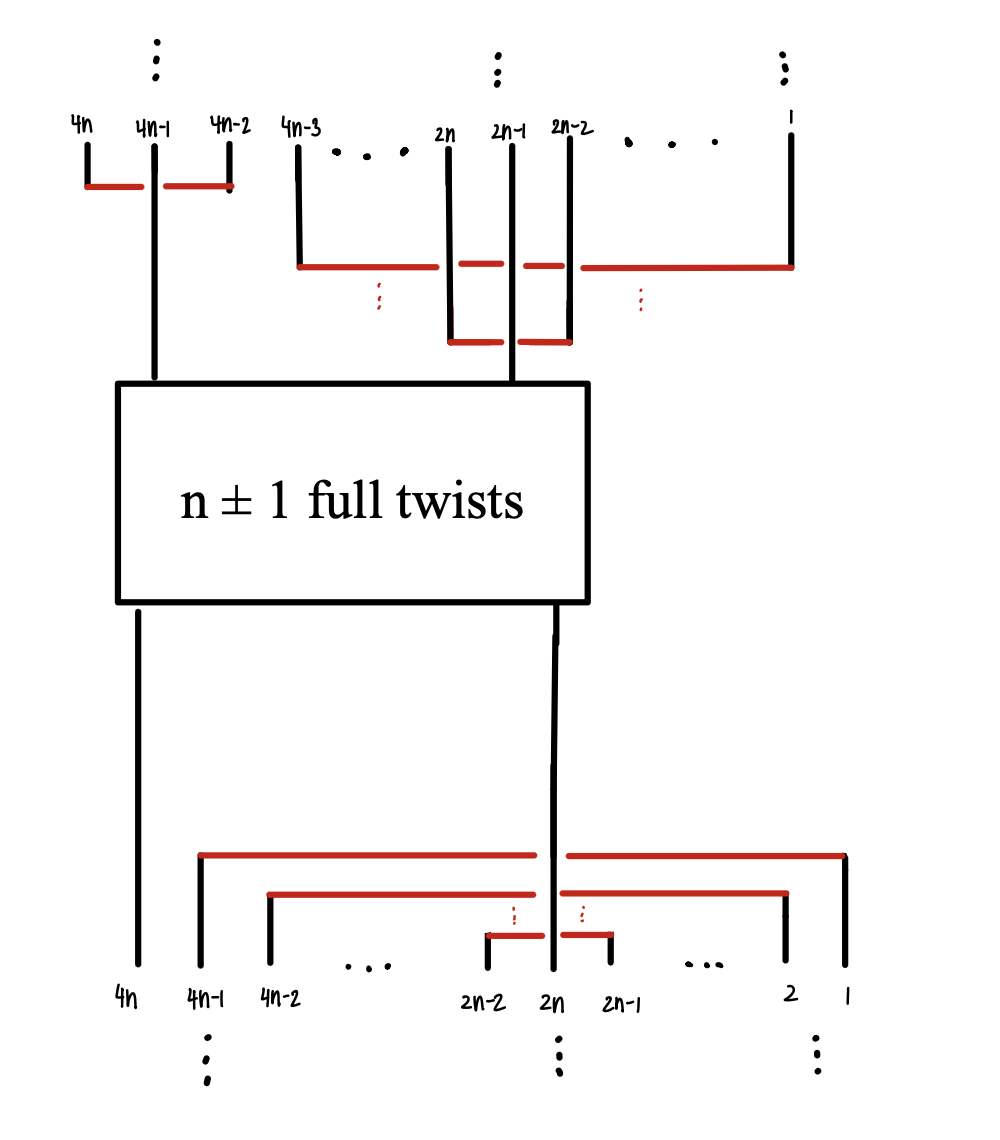}
        \caption{Moving Strand $1$ to position $4n$.}
        \label{fig:7}
    \end{figure}
    
    Additionally, each strand other than strands $2n $ and $4n$ (which are not paired up) belongs to precisely one pair and the pair completes two full rotations around the $4n - 1$ partial twists as it moves through them, eliminating exactly two partial twists in the process. The resulting diagram is shown in Figure \ref{fig:6}. Finally, pulling strand $1$ underneath the other strands to become strand $4n$ completes the isotopy, resulting in the diagram depicted in Figure \ref{fig:7}. Denote this knot $K_2$. Observe that $K_2$ is isotopic to $K_1$: the diagram in Figure \ref{fig:7} is simply a rotated and horizontally flipped version of Figure \ref{fig:3}.
    
    In Proposition 3.6 of \cite{longo2020infinitefamilycounterexamplesbatsons}, Longo proves that $K_1$ is the closure of the rational tangle $\tau \left(\frac{2n}{4n(n\pm 1) - 1} \right)$. According to Siebenman \cite{siebenman}, Casson, Gordon, and Conway showed all knots of this form are slice. One can verify this from Lisca’s classification of $2$-bridge knots \cite{Lisca_2007} to determine that $K_1$ is slice. It follows that $\gamma_4(T_{4n, (2n\pm 1)^2 + 4n-2}) \leq 2n-1$, and so $\gamma_4(T_{4n, (2n\pm 1)^2 + 4n-2}) = 2n-1$. 
\end{proof}

\section{Extending calculations to a new family of torus knots}

As noted in \cite{Jabuka_2021}, we can extend our family of examples by considering torus knots that, after a series of pinch moves, transform into the knots $T_{4n, (2n \pm 1)^2 + 4n-2}$. 
\begin{corollary}
    $\gamma_4(T_{4n + 2k, n(4n + 2k) - 1}) = 2n-1+k$ for all $n \geq 2$ and $k \geq 0$, and  $\gamma_4(T_{4n + 2k, (n+2)(4n + 2k) - 1}) = 2n-1+k$ for all $n \geq 1$ and $k \geq 0$. 
\end{corollary}

\begin{proof}

    \textbf{Case 1:} $\gamma_4(T_{4n + 2k, n(4n + 2k) - 1})$. 

    Using Lemma \ref{lem_pinch_move}, we analyze the pinch move applied to $\gamma_4(T_{4n + 2k, n(4n + 2k) - 1})$. The resulting system of congruences is: 
    \begin{align*}
        t \equiv -( n(4n + 2k) - 1)^{-1} &\pmod{4n+2k} \\ 
        h \equiv (4n+2k)^{-1} &\pmod{ n(4n + 2k) - 1} 
    \end{align*}
    This simplifies to: 
    \begin{align*}
        -t +1 \equiv 0 &\pmod{4n+2k} \\ 
        h \cdot (4n+2k) -1\equiv 0 &\pmod{ n(4n + 2k) - 1} 
    \end{align*}
    which has solutions $t = 1$ and $h = n$. 
    
    Thus, a pinch move reduces $T_{4n + 2k, n(4n + 2k) - 1}$ to $T_{4n + 2(k-1), n(4n + 2(k-1)) - 1}$. Repeating this $k$ times yields the knot $T_{4n, 4n^2 - 1} = T_{4n, (2n-1)^2 + 4n - 2}$, which satisfies $\gamma_4(T_{4n, 4n^2 - 1}) = 2n-1$ by Theorem \ref{thm_main}. By Proposition \ref{prop_bound}, we get $$ \gamma_4(T_{4n + 2k, n(4n + 2k) - 1}) \leq 2n-1+k $$ Combining this with the lower bound from Theorem \ref{1.7_binns}, we have $$\gamma_4(T_{4n + 2k, n(4n + 2k) - 1}) = 2n-1+k. $$ 

    \textbf{Case 2:} $\gamma_4(T_{4n + 2k, (n+2)(4n + 2k) - 1})$. 
    
    We apply the same strategy. Using Lemma \ref{lem_pinch_move}, we get: 
    \begin{align*}
        t \equiv -( (n+2)(4n + 2k) - 1)^{-1} &\pmod{4n+2k} \\ 
        h \equiv (4n+2k)^{-1} &\pmod{ (n+2)(4n + 2k) - 1} 
    \end{align*}
    which simplifies to 
    \begin{align*}
        -t + 1 \equiv 0 &\pmod{4n+2k} \\ 
        h \cdot (4n+2k) - 1 \equiv 0 &\pmod{ (n+2)(4n + 2k) - 1} 
    \end{align*}
    and has solutions $t = 1$ and $h = n+2$. 
    
    Thus, one pinch move sends $T_{4n + 2k, (n+2)(4n + 2k) - 1}$ to $T_{4n + 2(k-1), (n+2)(4n + 2(k-1)) - 1}$. After $k$ such moves, we reach $T_{4n, (n+2)(4n) - 1} = T_{4n, (2n+1)^2 + 4n-2}$. By applying Theorem \ref{thm_main}, Proposition \ref{prop_bound}, and Theorem \ref{1.7_binns}, we conclude $$\gamma_4(T_{4n + 2k, (n+2)(4n + 2k) - 1}) = 2n-1+k. $$
\end{proof}

\subsection*{Remark} By Theorem \ref{1.7_binns}, $\vartheta(T_{4n + 2k, n(4n + 2k) - 1}) 
 = \vartheta(T_{4n + 2k, (n+2)(4n + 2k) - 1}) = 2n+k. $ Thus all torus knots of this form are counterexamples to Batson's conjecture.

\bibliographystyle{alpha}
\bibliography{sample}
\end{document}